\documentclass{amsart}

\usepackage[outer=1.25in, inner=1.25in, top=1.25in, bottom=1.25in]{geometry}
\newtheorem{theorem}{Theorem}[section]

\newtheorem{proposition}[theorem]{Proposition}
\newtheorem{lemma}[theorem]{Lemma}

\newtheorem{conjecture}[theorem]{Conjecture}
\newtheorem{remark}[theorem]{Remark}
\newtheorem{example}[theorem]{Example}

\usepackage{hyperref,amssymb}
\begin{document}
	
	\title[Cardinality minimal bases]{On the cardinality of irredundant and minimal  bases of finite permutation groups}
	
	\author[F. Dalla Volta]{Francesca Dalla Volta}
	\address{Francesca Dalla Volta, Dipartimento di Matematica Pura e Applicata,\newline
		University of Milano-Bicocca, Via Cozzi 55, 20126 Milano, Italy} 
	\email{f.dallavolta@unimib.it}

	\author[F. Mastrogiacomo]{Fabio Mastrogiacomo}
	\address{Fabio Mastrogiacomo, Dipartimento di Matematica ``Felice Casorati", University of Pavia, Via Ferrata 5, 27100 Pavia, Italy} 
	\email{fabio.mastrogiacomo01@universitadipavia.it}

	\author[P. Spiga]{Pablo Spiga}
	\address{Pablo Spiga, Dipartimento di Matematica Pura e Applicata,\newline
		University of Milano-Bicocca, Via Cozzi 55, 20126 Milano, Italy} 
	\email{pablo.spiga@unimib.it}
	\thanks{The authors are members of the GNSAGA INdAM research group and kindly acknowledge their support.}
	\maketitle
	\begin{abstract}  
		Given a finite permutation group $G$ with domain $\Omega$, we associate two subsets of natural numbers to $G$, namely $\mathcal{I}(G,\Omega)$ and $\mathcal{M}(G,\Omega)$, which are the sets of cardinalities of all the irredundant and
		minimal bases of $G$, respectively. We prove that $\mathcal{I}(G,\Omega)$ is an interval of natural numbers, whereas $\mathcal{M}(G,\Omega)$ may not necessarily form an interval. Moreover, for a given subset of natural numbers $X \subseteq \mathbb{N}$, we provide some conditions on $X$ that ensure the existence of both intransitive and transitive groups $G$ such that 
		$\mathcal{I}(G,\Omega) = X$ and $\mathcal{M}(G,\Omega) = X$.
	\end{abstract}
	
	\section{Introduction}        	          
	Let $G$ be a permutation group on $\Omega$. A subset $\Lambda=\{\omega_1,\ldots,\omega_\ell\}$ of $\Omega$ is said to be a \textit{\textbf{base}} if the pointwise stabilizer $$G_{(\Lambda)}=G_{\omega_1,\ldots,\omega_\ell}$$
	equals the identity. Moreover, $\Lambda$ is said to be \textit{\textbf{minimal}} if no proper subset of it is a base. We denote with $b(G,\Omega)$ the smallest cardinality of a 
	(minimal) base of $G$ and with $B(G,\Omega)$ the maximum cardinality of a minimal base of $G$.
	
	Given an ordered sequence $(\omega_1,\ldots,\omega_\ell)$ of elements of $\Omega$, we study the associated stabilizer chain:
	$$G\ge G_{\omega_1}\ge G_{\omega_1,\omega_2}\ge \cdots \ge G_{\omega_1,\omega_2,\ldots,\omega_\ell}.$$
	If all the inclusions given above are strict, then the stabilizer chain is called \textit{\textbf{irredundant}}. If, furthermore, the group $G_{\omega_1,\ldots,\omega_\ell}$ is the identity, then the sequence $(\omega_1,\ldots,\omega_\ell)$ is called an \textit{\textbf{irredundant base}}. The size of the longest possible irredundant base is denoted $I(G,\Omega)$. Note that an irredundant base is not a base, because it is an ordered sequence and not a set.  However, each minimal base gives rise to an irredundant base. Therefore,
	$$b(G,\Omega)\le B(G,\Omega)\le I(G,\Omega).$$
	
	In this paper, given a finite permutation group $G$ with domain $\Omega$, we are interested in two subsets of integers associated to $G$ arising from irredundant bases and from  minimal bases. We let
	\begin{align*}
		\mathcal{I}(G,\Omega):=\{\ell\in\mathbb{N}\mid &\exists \omega_1,\ldots,\omega_\ell\in \Omega\hbox{ such that }(\omega_1,\ldots,\omega_\ell)
		\hbox{ is an irredudant base for }G\},\\
		\mathcal{M}(G,\Omega):=\{\ell\in\mathbb{N}\mid &\exists \omega_1,\ldots,\omega_\ell\in \Omega\hbox{ such that }\{\omega_1,\ldots,\omega_\ell\}
		 \hbox{ is an minimal base for }G\}.
	\end{align*}
	Our main results describe the subsets of $\mathbb{N}$ that arise as $\mathcal{I}(G,\Omega)$ and $\mathcal{M}(G,\Omega)$.
	\begin{theorem}\label{thrm:1}
		Let $G$ be a finite permutation group with domain $\Omega$. Then $\mathcal{I}(G,\Omega)$ is an interval of natural number, that is, $\mathcal{I}(G,\Omega)=\{b(G,\Omega),b(G,\Omega)+1,\ldots,B(G,\Omega)\}.$
	\end{theorem}
	\begin{theorem}\label{thrm:2}
		Let $X$ be an arbitrary non-empty subset of positive integers. Then, there exists a permutation group $G$ with domain $\Omega$ such that $X = \mathcal{M}(G,\Omega)$.
	\end{theorem}
	\begin{theorem}\label{thrm:3}
		Let $X$ be an interval of positive integers, not containing $1$. Then, there exists a finite transitive permutation group $G$ with domain $\Omega$ such that  $X=\mathcal{M}(G,\Omega)$ and a finite transitive permutation group $H$ with domain $\Delta$ such that $X = \mathcal{I}(H,\Delta)$.
	\end{theorem}
	Theorem~\ref{thrm:1} shows that the cardinalities of the irredundant bases of finite permutation groups form an interval of positive integers, and it somehow resembles a result of Tarski~\cite{Tarski}. Indeed, one of the main results in~\cite{Tarski} shows that the cardinalities of the irredundant generating sets of a finite group form an interval. Recall that a set of generators of a finite group $G$ is said to be irredundant if no proper subset of it generates $G$. This means that, if $G$ has irredundant generating sets of cardinality $x$ and $y$ with $x\le y$, then for every $z$ with $x\le z\le y$, $G$ admits an irredundant generating set of cardinality $z$. In this light, Theorem~\ref{thrm:1} is a permutation group analogue of Tarski's theorem. Despite the fact that the result of Tarski applies to general universal algebras and to general closure operations, we are not able to adapt the proof in~\cite{Tarski} to prove Theorem~\ref{thrm:1}.\\
	The proof of Theorem \ref{thrm:1} is due to Peter Cameron. It should be noted that this proof remains unpublished, except for its appearance on Cameron's blog\footnote{https://cameroncounts.wordpress.com/2023/04/15/bases-2/}. For the sake of completeness, we include the proof of Theorem~\ref{thrm:1} in Section \ref{sec1}.

	In light of this, Theorem~\ref{thrm:2} comes with a bit of a surprise, because it shows that any set of positive integers arises as $\mathcal{M}(G,\Omega)$ for some permutation group $G$. In particular, $\mathcal{M}(G,\Omega)$ is not an interval of natural numbers.
	
	The next question to explore is the behaviour of $\mathcal{M}(G,\Omega)$ when $G$ is a transitive group. Theorem~\ref{thrm:3} shows that every interval of positive integers, not containing the number $1$, can be realized as $\mathcal{M}(G,\Omega)$ for some transitive permutation group $G$ acting on its domain $\Omega$. It is important to note that if $1 \in \mathcal{M}(G,\Omega)$ and $G$ is transitive, then $\mathcal{M}(G,\Omega) = \{1\}$, because $G$ is regular on $\Omega$.	

	While this last result might suggest that $\mathcal{M}(G,\Omega)$ is always an interval when $G$ is transitive, we prove in Section \ref{SpecCase} that this is not universally true. In fact, we show that for certain non-interval subsets $X \subseteq \mathbb{N}$, there exists a transitive permutation group $G$ acting on its domain $\Omega$ such that $\mathcal{M}(G,\Omega) = X$. However, we encounter limitations when addressing the problem in a general context. For instance, we are unable to find an example with $\mathcal{M}(G,\Omega) = \{4,6\}$.	Nevertheless, we propose the following conjecture.
	\begin{conjecture}\label{conj1}
		Let $X \subseteq \mathbb{N}$, with $1 \notin X$. Then, there exists a transitive permutation group $G$ on $\Omega$ with $\mathcal{M}(G,\Omega) = X.$
	\end{conjecture}

	We conclude this introductory section observing that we know very little on the structure of $\mathcal{M}(G,\Omega)$ when $G$ is primitive on $\Omega$. 
	Indeed, in the case of primitive groups, we have no example where $\mathcal{M}(G,\Omega)$ is not an interval. This leads us to propose the following conjecture.
	\begin{conjecture}\label{conj2}
		Let $G$ be a primitive permutation group on $\Omega$. Then, $\mathcal{M}(G,\Omega)$ is an interval of natural numbers.
	\end{conjecture}
We do not believe that, for each interval $X$ of natural numbers, there exists a primitive group $G$ on $\Omega$ with $\mathcal{M}(G,\Omega)=X$. We give some evidence of this in Section~\ref{last}. 
	
A permutation group $G$ on $\Omega$ is said to be IBIS (\textit{\textbf{Irredundant Bases of Invariant Size}}) if all irredundant bases have cardinality $b(G,\Omega)$, that is, $\mathcal{I}(G,\Omega)=\{b(G,\Omega)\}$. Cameron and Fon-der-Flaass~\cite{CF} (see also~\cite[Section~4.14]{Peter}) have proved that, in a finite permutation group, the following conditions are equivalent:
\begin{itemize}
\item all irredundant bases have the same size;
\item the irredundant bases are invariant under re-ordering;
\item the irredundant bases are the bases of a matroid.
\end{itemize}	
Observe that, for every permutation group $G$ on $\Omega$, we have $\mathcal{M}(G,\Omega)\subseteq\mathcal{I}(G,\Omega)$. Therefore, we say that $G$ is MiBIS (\textit{\textbf{Minimal Bases of Invariant Size}}) if all minimal bases have cardinality $b(G,\Omega)$, that is, $\mathcal{M}(G,\Omega)=\{b(G,\Omega)\}$. It is evident that each IBIS group is also a MiBIS group, but the reverse is far from being true.\footnote{The action of $G = \mathrm{GL}_4 (2)$ on the $2$-subspaces of $\mathbb{F}_2^4$  has irredundant bases of size $4$ and $5$; however,
it can be verified with the auxiliary help of a computer that every minimal base of
$G$ has cardinality $4$. Therefore, in this action, $G$ is MiBIS, but not IBIS.} This leads us to question whether MiBIS groups possess a geometric characterization in the same vein as Cameron and Fon-der-Flaas. In addition, Lucchini, Morigi, and Moscatiello have established a theorem that reduces the problem of classifying finite primitive IBIS groups $G$ to cases where the socle of $G$ is either abelian or non-abelian simple. We also wonder whether a similar reduction applies to MiBIS groups.

	\section{Proof of Theorems~$\ref{thrm:1}$ and~\ref{thrm:2}} \label{sec1}
\begin{proof}[Proof of Theorem~$\ref{thrm:1}$]Let $G$ be a transitive group with domain $\Omega$. Let $a = (a_1 , \ldots , a_x )$ and $b = (b_1 ,\ldots , b_{y})$  be irredundant bases for $G$.
	Suppose, for a contradiction, that there is no irredundant base of length $z$,
	for some $z$ with $x < z < y$.
	Consider the tuple $(b_1, \ldots , b_{z-x} , a_1 , \ldots , a_x )$ of points in the domain $\Omega$. This is a base, since it contains $a$. By assumption it is redundant because it is of length $z$. None of the $b$s is redundant, so we
	must have to delete some of the $a$s. The number remaining is, say $x'$; we
	have $x' < x$ (since there is no irredundant base of length $z$) and $x' > 0$ (since
	the $b$s are irredundant). Let $a'$ be the tuple of $a$s of length $x'$.
	Now consider the tuple $(b_1 , \ldots , b_{z-x'} )$ with $a'$ appended. Again this is a
	base, and must be redundant, so there is a subtuple $a''$ of length $x''$ which
	forms a base with these $b$s. Again we have $0 < x'' < x'$.
	This process can be continued ad infinitum to give an infinite descending
	sequence of natural numbers less than $x$, a contradiction.
\end{proof}

	Prior to demonstrating Theorem~\ref{thrm:2} in its complete generality,  we begin by providing an illustrative example that exemplifies the fundamental concepts of the 
	proof, when $X$ is defined as $\{1,3,5,7\}$. The idea is to consider a sequence of regular elementary abelian groups with disjoint orbits, and combining their generators in order to obtain the desired group.
	
	Consider the group $$G_1 = \langle g_{11},\dots, g_{17}\rangle$$ to be a regular elementary abelian group of order $2^7$ acting on $\Delta_1$ (to construct it, just consider the elementary abelian group of order $2^7$ in its action on itself by right multiplication). Then, take the group $$ G_2 = \langle g_{21},g_{22},g_{23} \rangle $$ to be a regular elementary abelian group of order $2^3$ acting on $\Delta_2$. Then, take $$ G_3 = \langle g_{31},\dots, g_{35} \rangle $$ to be a regular elementary abelian group of order $2^5$ acting on $\Delta_3$. Suppose that the $\Delta_i$s are mutually disjoint. Finally, take $g_{41},\dots,g_{47}$ to be transpositions, each one acting on $\Delta_{41},\dots,\Delta_{47}$ respectively, disjoint from each other, and also disjoint from $\Delta_1, \Delta_2, \Delta_3$.
	Now define 
	\begin{align*}
		G = \langle 
		&h_1=g_{11}g_{21}g_{31}g_{41}, \\ 			&h_2=g_{12}g_{22}g_{32}g_{42}, \\ 
		&h_3=g_{13}g_{23}g_{33}g_{43}, \\
		&h_4=g_{14}g_{34}g_{44}, \\
		&h_5=g_{15}g_{35}g_{45}, \\
		&h_6=g_{16}g_{46},\\
		&h_7=g_{17}g_{47} 
		\rangle,
	\end{align*}
	and consider the action of $G$ on $\Omega=\Delta_1 \cup \Delta_2 \cup \Delta_3 \cup \Delta_{4,1} \cup \dots \cup \Delta_{4,7}$. We claim that
	$$\mathcal{M}(G,\Omega)=\{1,3,5,7\}.$$ Note that $G$ is an elementary abelian group of order $2^7$. Moreover, note that the $g_{i,j}$'s appearing in $h_j$ commute among themselves, because their supports are disjoint. 
	
	Let $y_1 \in \Delta_1$. Each of the $h_j$s moves $y_1$, because every generator of $G$ involves one generator of $G_1$ and because $G_1$ is regular on $\Delta_1$. Hence, the stabilizer of $y_1$ is the identity subgroup, that is, $G_{y_1}=1$. Therefore $G$ has a base of cardinality $1$ and $1\in\mathcal{M}(G,\Omega)$. In particular, for the rest of the argument, we may consider minimal bases containing no point from $\Delta_1$.
	
	 Let $y_2 \in \Delta_2$. This point is moved only by the generators of $G$ which involve generators of $G_2$, namely $h_1,h_2,h_3$. This implies that $$G_{y_2} = \langle h_4,h_5,h_6,h_7 \rangle,$$ and this acts non-trivially on $\Delta_1 \cup \Delta_{4,4} \cup \Delta_{4,5}\cup\Delta_{4,6} \cup \Delta_{4,7}$. So in order to construct a minimal base which contains $y_2$ we need to choose the other points from $\Delta_1 \cup \Delta_{4,4} \cup \Delta_{4,5}\cup\Delta_{4,6} \cup \Delta_{4,7}$. Since we are excluding the elements of $\Delta_1$, the other points are from $\Delta_{4,4} \cup \Delta_{4,5}\cup\Delta_{4,6} \cup \Delta_{4,7}.$  Observe now that, if $i\in \{1,\dots, 7\}$ and $y_{4,i} \in \Delta_{4,i}$, then $$G_{y_{4,i}} = \langle h_1,\dots,h_{i-1},h_{i+1},\dots,h_7 \rangle. $$ So, the stabilizer of $y_2$ and $y_{4,4}$ is
	\[
		G_{y_2} \cap G_{y_{4,4}} = \langle h_5,h_6,h_7 \rangle.
	\]
	It is now clear that $$\{y_2,y_{4,4},y_{4,5},y_{4,6},y_{4,7}\}$$ is a minimal base of cardinality $5$ for $G$ and $5\in\mathcal{M}(G,\Omega)$. In particular, for the rest of the argument, we may consider minimal bases containing no point from $\Delta_2$.
	
	Analogously, if $y_3 \in \Delta_3$, then $$G_{y_3} = \langle h_6, h_7\rangle,$$ which acts non-trivially on $\Delta_1 \cup \Delta_{4,6} \cup \Delta_{4,7}$. So in order to construct a minimal base of $G$ which contains $y_3$, we need to choose the other points from this set. Since we are excluding the elements of $\Delta_1$, the other points are from $ \Delta_{4,5}\cup\Delta_{4,6} \cup \Delta_{4,7}.$ As before, we can only choose an arbitrary point from $\Delta_{4,6}$ and an arbitrary point from $\Delta_{4,7}$, so that we have a minimal base of cardinality $3$ and hence $3\in\mathcal{M}(G,\Omega)$. 
	
	Finally, we can construct a minimal base using points only from $\Delta_{4,i}$ for $i=1,\dots,7$. For what we have observed before, in order to construct a minimal base we need to take an arbitrary point from each one of the orbits. In conclusion, this gives a minimal base of cardinality $7$. This marks the completion of our introductory example's proof.

\smallskip	
	
	To prove Theorem \ref{thrm:2}, we need some basic preliminaries. Let $G$ and $H$ be two permutation groups acting on $\Delta$ and $\Lambda$ respectively. We can suppose that $\Delta$ and $\Lambda$ are disjoint, eventually renaming the elements. Then, $G \times H$ acts on $\Delta \cup \Lambda$ as follows: given $x \in \Delta \cup \Lambda$ and $(g,h) \in G \times H$, we set 
	\[
		x^{(g,h)} = 
		\begin{cases}
			x^g \, \text{ if } x \in \Delta, \\
			x^h \, \text{ if } x \in \Lambda.
		\end{cases}
	\]
	Note that this action is intransitive even if $G$ and $H$ are transitive. This is in contrast with another action of the direct product, described in Section \ref{SectionThmMinBase}.\\
	From the definition of the action, we have that, for $x \in \Delta \cup \Lambda$,
	\[
		(G \times H)_x = 
		\begin{cases}
			G_x \times H \, \text{ if } x \in \Delta, \\
			G \times H_x \, \text{ if } x \in \Lambda.
		\end{cases}
	\]
	It follows that, if $B_G$ is a minimal base for $G$ and $B_H$ is a minimal base for $H$, then $B_G \cup B_H$ is a minimal base for the action of $G \times H$ on $\Delta \cup \Lambda$. Moreover, every minimal base of $G \times H$ is of this form. Indeed, suppose that $B \subseteq \Delta \cup \Lambda$ is a base for $G \times H$. Define $B_X = X \cap B$ for $X \in \{\Delta, \Lambda\}$, so that $B = B_\Delta \cup B_\Lambda$. Then,
	\[
		1 = (G \times H)_{(B)} = G_{(B_\Delta)} \times G_{(B_\Lambda)},
	\]
	So, $B_\Delta$ and $B_\Lambda$ are bases for $G$ and $H$ respectively. They are also minimal: otherwise we could find, for example, $\tilde{B}_\Delta \subset B_\Delta$ such that $(G \times H)_{(\tilde{B}_\Delta \cup B_\Lambda)}=1$, against the minimality of $B$. In conclusion, we have the following lemma.
	\begin{lemma}
		\label{lemmaSumBase}
		Let $G$ and $H$ be permutation groups on $\Delta$ and $\Lambda$ respectively. Suppose that $\Delta$ and $\Lambda$ are disjoint and consider the action of $G \times H$ on $\Delta \cup \Lambda$. Then
		\[
			\mathcal{M}(G \times H,\Delta\cup \Lambda) = \{a+b \, | \, a \in \mathcal{M}(G,\Delta), \, b \in \mathcal{M}(H,\Lambda)\}.
		\]
	\end{lemma}

\begin{proof}[Proof of Theorem~$\ref{thrm:2}$]Let $X = \{x_1,x_2,\dots,x_n\}\subseteq N$, with $x_1 < x_2 < \dots < x_n$. \\
	Firstly, suppose that $x_1=1$, and fix a prime number $p$. Define 
	\[
		G_1 = \langle g_{1,1},\dots,g_{1,x_n} \rangle
	\]
	to be a regular elementary abelian group of order $|G_1| = p^{x_n}$ with domain $\Delta_1$. 
	Moreover, for each $j=2,\dots,n-1$, define
	\[
		G_j = \langle 	g_{j,1},\dots,g_{j,x_n-x_{n-j+1}+1} \rangle
	\]
	to be a regular elementary abelian group of order $ |G_j| = p^{x_n-x_{n-j+1}+1}$ with domain $\Delta_j$. Suppose that $\Delta_j \cap \Delta_i = \emptyset$ for $i \neq j$. 
	Additionally, for $i> x_n-x_{n-j+1}+1$, define $g_{j,i}=1$.

	Finally, let $g_{n,1},\dots, g_{n,x_n}$ be cycles of length $p$ with pair-wise disjoint supports $\Delta_{n,1}, \dots, \Delta_{n,x_n}$, which in turn are also disjoint from $\Delta_i$, for $i=1,\dots,n-1$.
	\\
	Consider now the group
	\begin{align*}
		G = \langle &g_{1,1} g_{2,1} \cdots g_{n,1},\,\,g_{1,2} g_{2,2} \cdots g_{n,2},\,\, g_{1,3} g_{2,3} \cdots g_{n,3},\,\, \dots\,\, g_{1,x_n}g_{2,x_n} \cdots g_{n,x_n} \rangle
	\end{align*}
	acting on $\Omega=\Delta_1 \cup \dots \cup \Delta_{n-1} \cup \Delta_{n,1} \cup \dots \cup \Delta_{n,x_n}$.
	We claim that $\mathcal{M}(G,\Omega) = X$.

	Let $i \in {1,\dots,n-1}$ and let $\delta \in \Delta_i$. Observe that, by construction, $G_i$ acts regularly on $\Delta_i$, and hence the stabilizer of $\delta$ in $G$ fixes pointwise $\Delta_i$. Moreover, observe that the stabilizer of $\delta$ in $G$ fixes all the points in the orbits $\Delta_{2},\dots, \Delta_{i}$. Indeed, the number of generators of $G_i$ is greater than the number of generators of $G_j$ for $1< j<i$. Thus, the generators of $G$ which involve generators of $G_i$, also involve the ones of $G_j$. This implies that,  whenever we stabilize $\delta$, we also stabilize the points in $\Delta_2,\dots, \Delta_{i}$. Moreover, note that the points of the orbits $\Delta_{n,1},\dots,\Delta_{n,x_n-x_{n-i+1}+1}$ also are fixed by the stabilizer of $\delta$.
	\\
	Thus, in order to construct a minimal base for $G$, if we take a point from $\Delta_i$, we are forced to choose the other points in the orbits $\Delta_{n,x_n-x_{n-i+1}+2},\dots,\Delta_{n,x_n}$.
	
	We are now ready to construct the minimal bases of $G$. Firstly, $G$ has a base of cardinality $1$, given by an arbitrary point in $\Delta_1$. Indeed, the generators of $G_1$ appear in every generator of $G$.\\ 
	Let now $i < n$, $y \in \Delta_{n-i+1}$, and $y_j \in \Delta_{n,j}$, for $j=x_n-x_i+2,\dots,x_n$. We claim that $$B=\{y,y_{x_n-x_i+2},\dots,y_{x_n}\}$$ is a minimal base of cardinality $x_i$. By the argument in the previous paragraph, the only points that are moved by the stabilizer of $y$ are the ones in the orbits $\Delta_{n,x_n-x_i+2}, \dots, \Delta_{n,x_n}$. So by fixing a point in each of these orbits, we get the identity subgroup. Moreover, by noting that  $$g_{n-i+1,1}\cdots g_{n-i+1,n} \in G_{y_{x_n-x_i+2},\dots,y_{x_n}}$$ and $$g_{1,y_j} \cdots g_{n,y_j} \in G_{y,y_{x_n-x_i+2},\dots,y_{j-1},y_{j+1},\dots,y_{x_n}},$$ we conclude that $B$ is a minimal base.\\ Finally, a base of cardinality $x_n$ is given by $B = \{ y_1, \dots, y_{x_n}\}$, where $y_{j} \in \Delta_{n,j}$.  So far, we have proved that $X \subseteq \mathcal{M}(G,\Omega)$. 
	
	Let now $B$ be a minimal base for $G$ with $1 < |B| < x_n$. Then $B$ must contain a point from one of the orbits $\Delta_2,\dots, \Delta_{n-1}$. If not, either it has a point from $\Delta_1$  or it has points only from $\Delta_{n,1},\dots,\Delta_{n,x_n}$. In the former case, the stabilizer of the point in $\Delta_1$ would be the identity, so it is not possible. In the latter, since $|B| < x_n$, there is an orbit $\Delta_j$ which does not contain points of $B$, and so $G_{(B)} \neq 1$. From the remark above, $B$ is a base of the same form of the ones we have constructed in the previous paragraph. In conclusion, $\mathcal{M}(G,\Omega) = X$.
	
	Suppose now $x_1 > 1$, and take $Y=\{1, x_2-x_1+1, \dots, x_n-x_1+1\}$. Then, from the first part of the proof, there exists a permutation group $H$, acting on $\Lambda$, such that $\mathcal{M}(H,\Lambda) = Y$. Let $\mathrm{Sym}(x_1)$ be the symmetric group on $x_1$ symbols acting on $\{1,\dots,x_1\}$, and consider the group
	\[
	G = \mathrm{Sym}(x_1) \times H
	\]
	acting on the disjoint union $\Omega$ of $\{1,\dots, x_1\}$ and $\Lambda$. By Lemma~\ref{lemmaSumBase},
		$$\mathcal{M}(G,\Omega) =\{a+b\mid a\in \mathcal{M}(\mathrm{Sym}(x_1),\{1,\ldots,x_1\}), b\in \mathcal{M}(H,\Lambda)\}=\{x_1-1+b\mid b\in Y\}=X.\qedhere$$
	\end{proof}
	
	\section{Theorem~$\ref{thrm:3}$} \label{SectionThmMinBase}
	Our proof of Theorem \ref{thrm:3} is based on the product action of two (or more) permutation groups. 
	Let $G, H$ be two groups acting on $\Delta$ and $\Lambda$ respectively. Given $(g,h) \in G \times H$ and $(\delta, \lambda) \in \Delta \times \Lambda$, we define
	\[
		(\delta,\lambda)^{(g,h)} = (\delta^g, \lambda^h).
	\]
	It is straightforward to see that this defines an action of $G \times H$ on $\Delta \times \Lambda$. We refer to this action as the product action. Moreover, if $G$ and $H$ act transitively on $\Delta$ and $\Lambda$ respectively, then so does $G\times H$ on $\Delta \times \Lambda$.

Irredundant bases for this particular type of action have been previously examined in~\cite{GLSp}. In this context, as part of our effort to establish Theorem~\ref{thrm:3}, we present a result concerning minimal bases in the product action, a topic that in our opinion carries its own significance.

	Let $G$ be a permutation group on $\Omega$, and let  $\Lambda$ be a subset of $\Omega$. We say that $\Lambda$ is an indipendent set if its pointwise stabilizer is not equal to the pointwise stabilizer of any proper subset. The height of $G$, $H(G,\Omega)$, is defined to be the maximum size of an indipendent set of $G$. Note that a minimal base is in particular an indipendent set. In \cite[Lemma 2.5]{GLSp}, the authors prove that for the product action of two permutation groups $G$ with domain $\Delta$ and $H$ with domain $\Lambda$, 
	\[
		H(G \times H,\Delta\times\Lambda) \leq H(G,\Delta) + H(H,\Lambda).
	\]
	In particular, we have the following lemma.
	\begin{lemma}
		Let $G,H$ acting on $\Delta$, $\Lambda$ respectively. Then,
			$B(G\times H,\Delta\times\Lambda) \leq B(G,\Delta) + B(H,\Lambda).$
	\end{lemma}

	To analyze the minimal bases of $G \times H$ in its action on $\Delta\times \Lambda$, we use the following construction.
	Let $\{(\delta_1,\lambda_1), \dots , (\delta_k, \lambda_k)\}$ be a minimal base for the product action of $G \times H$. We define the $\{0,1\}$-vectors $v^G = (v^G_1,\dots, v^G_k)$ and $v^H = (v^H_1,\dots,v^H_k)$ where
	\[
	 v^G_i =
		\begin{cases}
			1  \text{ if } \exists g \in G \, : \, \delta_i^g \neq \delta_i, \delta_j^g = \delta_j \, \forall j \neq i ,\\
			0 \text{ otherwise,}\\
		\end{cases}
	\]
	and
	\[
	v^H_i =
	\begin{cases}
		1  \text{ if } \exists h \in H \, : \, \lambda_i^h \neq \lambda_i, \lambda_j^h = \lambda_j \, \forall j \neq i ,\\
		0 \text{ otherwise.}\\
	\end{cases}
	\]
	Roughly speaking, these vectors measure how many elements of $\{\delta_1,\dots,\delta_k\}$ and $\{\lambda_1,\dots,\lambda_k\}$ are necessary to form a minimal base for $G$ and $H$ respectively. Observe that $v^G$ and $v^H$ depend on the group $G\times H$ in its action on $\Delta\times\Lambda$ and also on the minimal base $\{(\delta_1,\lambda_1), \dots , (\delta_k, \lambda_k)\}$: for not making the notation too cumbersome, we omit the dependency on the base for denoting $v^G$ and $v^H$.

	Finally, for $X \in \{G,H\}$, we let
	\[
		n_X = |\{i \in \{1,\dots,k\} \, | \, v^X_i = 1 \}|.
	\]
	So $n_X$ is the number of $1$s appearing in the vector associated to $X$.
	\begin{remark}{\rm If $v_i^G = 0$ and $g \in G$, then two things can happen: either $\delta_i^g = \delta_i$, or there exists $j \in \{1,\dots,k\}$, $j \neq i$, such that $\delta_j^g \neq \delta_j$.}
	\end{remark}
	
	\begin{lemma}
		\label{lemmaVect}
		Let $G,H$ be acting on $\Delta$, $\Lambda$ respectively and let  $\{(\delta_1,\lambda_1), \dots , (\delta_k, \lambda_k)\}$ be a minimal base for the product action of $G \times H$. Then, the following hold.
		\begin{enumerate}
			\item\label{item1} $n_G \leq B(G,\Delta)$ and $n_H \leq B(H,\Lambda)$.
			\item\label{item2} There exists no $j\in \{1,\ldots,k\}$ with $v_j^G=v_j^H=0$. In particular, $$n_G + n_H \geq k.$$
			\item\label{item3} If $v^G_i = v^G_j = 1$, then $\delta_i \neq \delta_j$. Similarly, if $v^H_i = v^H_j = 1$, then $\lambda_i \neq \lambda_j$.
		\end{enumerate}
		Additionally, if $k=B(G,\Delta)+B(H,\Lambda)$, then the following hold.
		\begin{enumerate}
			 \setcounter{enumi}{3}
			\item\label{item4} $n_G = B(G,\Delta)$ and $n_H = B(H,\Lambda)$.
			
			\item\label{item5} For all $i=1,\dots,k$, $v^G_i \neq v^H_i$.
		\end{enumerate}		
	\end{lemma}
	\begin{proof}
	We prove~\eqref{item1}.	 Take $\Gamma = \{\delta_{i_1},\dots,\delta_{i_{n_G}}\}$ to be the subset of $\{\delta_1, \dots, \delta_k\}$  for which $v^G_i=1$. For each $\delta \in \Gamma$, there exists $g\in G \setminus \{1\}$ which moves $\delta$ while fixing all the other points, so that $g \in G_{(\Gamma \setminus \{\delta\})}\setminus \{1\}$. This implies that $\Gamma$ is a minimal base for $G$ of cardinality $n_G$. Therefore, $n_G \leq B(G,\Delta)$. The proof for $H$ is analogous. 
	
  We prove~\eqref{item2}.			 Suppose there exists $j \in \{1,\dots,k\}$ such that $v^G_j = v^H_j = 0$. Consider the set
			\[
				\tilde{B}=\{(\delta_1,\lambda_1),\dots,(\delta_{j-1},\lambda_{j-1}),(\delta_{j+1},\lambda_{j+1}),\dots,(\delta_k,\lambda_k)\}.
			\]
			Let $\Gamma = \{\delta_1,\dots,\delta_k\}$. Clearly, $G_{(\Gamma)} = 1$. We claim that $G_{(\Gamma\setminus \{\delta_j\})} = 1$. Let $g \in G_{(\Gamma\setminus \{\delta_j\})}$. Since $v^G_j = 0$ and  $g$ does not fix $\delta_j$, there exists $i\neq j$ such that $\delta_i^g \neq \delta_i$. This is impossible, because $g \in G_{(\Gamma \setminus \{\delta_j\})}$. Thus, $g$ must fix $\delta_j$, so  $g=1$. The same happens for $H$, and this implies that $\tilde{B}$ is a base for $G \times H$ in its product action on $\Delta\times \Lambda$, contradicting the minimality of $B$.
			
			We prove~\eqref{item3}. If $\delta_i=\delta_j$, then there is no $g$ which fixes $\delta_i$ and moves all the other points, as it would also fix  $\delta_j$. So $v^G_i = v^G_j = 0$.
			
			Part~\eqref{item4} follows from~\eqref{item1} and~\eqref{item2}.  
			
			We prove~\eqref{item5}.  Suppose that there exists $i$ such that $v^G_i = v^H_i$. From~\eqref{item2}, they are both equal to $1$. Since $k = B(G,\Delta) + B(H,\Lambda) = n_G + n_H$, there exists an index $j$ such that $v^G_j = v^H_j = 0$, but this contradicts~\eqref{item2}.
	\end{proof}

	From Lemma~\ref{lemmaVect}, when $B(G\times H , \Omega \times \Lambda) = B(G,\Omega) + B(H,\Lambda)$, the vectors $v^G$ and $v^H$ associated to the minimal base of maximal cardinality have $B(G,\Omega)$ and $B(H,\Lambda)$ coordinates equal to $1$ respectively, moreover, $v^G_i+v^H_i =1$ $\forall i$.
	\begin{remark}\label{remark}{\rm
		If $\{(\delta_1,\lambda_1),\dots,(\delta_k,\lambda_k)\}$ is a base for the product action, then $\{\delta_1,\dots,\delta_k\}$ is a base for $G$. Moreover, if we remove from $\{\delta_1,\ldots,\delta_k\}$ one $\delta_i$ with $v^G_i=0$, we obtain another base for $G$. Furthermore, if we remove from $\{\delta_1,\ldots,\delta_k\}$ all $\delta_i$s with $v^G_i=0$, we obtain a minimal base for $G$.}
	\end{remark}
	\begin{theorem}\label{thrm:41}
		Let $G,H$ be two permutation groups acting on $\Delta$, $\Lambda$ respectively. Then, the set $\mathcal{M}(G\times H,\Delta\times\Lambda)$ forms an interval of natural numbers, and
		\[
			\mathcal{M}(G \times H,\Delta\times\Lambda) = \{\max(b(G,\Delta),b(H,\Lambda)), \dots, B(G,\Delta)+B(H,\Lambda) - \varepsilon\},
		\]
		where $\varepsilon \in \{0,1,2\}$.
	\end{theorem}
	\begin{proof}
We start by proving that, for each $i = \max(b(G,\Delta),b(H,\Lambda)), \dots , B(G,\Delta) + B(H,\Lambda) -2$, there exists a minimal base for $G\times H$ of cardinality $i$. To do this, it is sufficient to prove that, for each minimal base of $G$ of size $a$, and for each minimal base of $H$ of size $b$, there exists a minimal base of $G\times H$ size $k$ for each $k$ between $\max(a,b)$ and $a+b-2$. Without loss of generality, we suppose that $a \le b$.

Let $\{\delta_1,\dots,\delta_{a}\}$ be a minimal base of $G$ of cardinality $a$ and let $\{\lambda_1,\dots,\lambda_{b}\}$ be a minimal base of $H$ of cardinality $b$. We argue geometrically: we draw the elements of $\Delta\times \Lambda$ as a grid where the columns are labelled by the elements of $\Delta$ and the rows are labelled by the elements of $\Lambda$. Moreover, we order columns and rows  in such a way that the two bases of $G$ and $H$ are in the top left corner. To form a minimal base for $G\times H$, we need to choose points of the grid.
		
		Note that in order to ensure that our grid points form a minimal base for $G\times H$, we must avoid specific geometric configurations. Indeed, suppose that we choose $(\delta_1,\lambda_1)$ and $(\delta_2,\lambda_2)$, with $\delta_1\ne\delta_2$ and $\lambda_1\ne \lambda_2$, and consider the stabilizer in $G\times H$ of these two points. Take $(g,h)$ in this stabilizer. Then 
		\[
		(\delta_1,\lambda_2)^{(g,h)}= (\delta_1^g,\lambda_2^h) = (\delta_1,\lambda_2),
		\]
		and the same for $(\lambda_2,\delta_1)$. This shows that, if $(g,h)$ fixes $(\delta_1,\lambda_1)$ and $(\delta_2,\lambda_2)$, then $(g,h)$ fixes the four points of the square 
		 $(\delta_1,\lambda_1)$, $(\delta_1,\lambda_2)$, $(\delta_2,\lambda_1)$ and $(\delta_2,\lambda_2)$.
		 Therefore, in order to achieve a minimal base for $G\times H$, we must avoid creating triangles or squares within our grid. 
		 
		Let $k$ be between $\max(a,b)$ and $a+b-2$. To construct a base of size $k$, proceed in this way.
		\begin{itemize}
			\item Start from the point $(\delta_1,\lambda_1)$ and choose $a+b-k-1$ points in diagonal: 
			\[
			(\delta_1,\lambda_1), (\delta_2,\lambda_2), \dots, (\delta_{a+b-k-1},\lambda_{a+b-k-1}).
			\]
			\item Now take $b-1$ vertical points:
			\[
			(\delta_{a+b-k},\lambda_{a+b-k-1}),\dots,(\delta_{a-1},\lambda_{a+b-k-1}).
			\]
			\item Now take the point $(\delta_a,\lambda_{a+b-k})$, and then continue  vertically for $k-a$ steps:
			\[
			(\delta_a, \lambda_{a+b-k+1}),\dots, (\delta_a,\lambda_b).
			\]
		\end{itemize}
		In conclusion, the minimal base for $G\times H$ of cardinality $k$ is
		\begin{align*}
			&(\delta_1,\lambda_1), (\delta_2,\lambda_2), \dots, (\delta_{a+b-k-1},\lambda_{a+b-k-1}),\\
			&	(\delta_{a+b-k},\lambda_{a+b-k-1}),\dots,(\delta_{a-1},\lambda_{a+b-k-1}), \\
			&(\delta_a,\lambda_{a+b-k}),	(\delta_a, \lambda_{a+b-k+1}),\dots, (\delta_a,\lambda_b).
		\end{align*}
		This is a base, because it fills the two bases of $G$ and $H$, and it is minimal, because we have not drawn triangles in our grid.
		
		From Remark~\ref{remark}, we deduce
		$$b(G\times H,\Delta\times\Lambda)=\max(b(G,\Delta),b(G,\Lambda)).$$
Summing up, so far, we have shown that
$$\mathcal{M}(G\times H,\Delta\times \Lambda)\supseteq\{
\max(b(G,\Delta),b(H,\Lambda)),\ldots,B(G,\Delta)+B(H,\Lambda)-2\}.$$
Therefore, if $B(G\times H,\Delta\times \Lambda)=B(G,\Delta)+B(H,\Lambda)-2$, then the theorem follows immediately with $\varepsilon=2$. Similarly, if 	$B(G\times H,\Delta\times \Lambda)=B(G,\Delta)+B(H,\Lambda)-1$, then the theorem follows with $\varepsilon=1$. Since by Lemma~\ref{lemmaVect} part~\eqref{item2} $B(G\times H,\Delta\times \Lambda)\le B(G,\Delta)+B(H,\Lambda)$, the only case that requires special care is when $B(G\times H,\Delta\times \Lambda)=B(G,\Delta)+B(H,\Lambda)$. Indeed, in this latter case, we need to prove that there exists also a minimal base of cardinality $B(G,\Delta) + B(H,\Lambda)-1$. 

For, let 
		\[
				B = \{(\delta_1,\lambda_1),\dots,(\delta_k,\lambda_k)\}
		\]
		be a minimal base with $k= B(G,\Delta) + B(H,\Delta)$. By Lemma~\ref{lemmaVect}, there exists an index $i$ such that the vectors $v^G$ and $v^H$ satisfy $v^G_i =0,\, v^G_{i+1} = 1,\, v^H_i=1, v^H_{i+1} = 0$. Consider now the subset
		\[
				B' = \{(\delta_1,\lambda_1),\dots,(\delta_{i-1},\lambda_{i-1}),(\delta_{i+1},\lambda_{i}),(\delta_{i+2},\lambda_{i+2}),\dots,(\delta_k,\lambda_k)\}.
		\]
		We claim that this is a minimal base for the action of $G\times H$ on $\Delta\times \Lambda$. \\
		\textit{Base}: consider $\Xi = \{\delta_1,\dots, \delta_{i-1},\delta_{i+1},\dots,\delta_k\}$. Since $v^G_i = 0$, there exists a subset of $\Xi$ which is a base for $G$, so $\Xi$ is a base for $G$. The same happens for $\Xi = \{\lambda_1,\dots,\lambda_i,\lambda_{i+2},\dots,\lambda_k\}$. This shows that $B'$ is a base for $G\times H$.
		\\
		\textit{Minimal}: Consider $\tilde{B} = B' \setminus \{(\delta_j,\lambda_j)\}$, for $j=1,\dots,i-1,i+2,\dots,k$. Then, at least one of $v^G_j$ or $v^H_j$ is equal to $1$. If for example $v^G_j = 1$, then there exists $g \in G$  such that $\delta_j^g \neq \delta_j$ and $\delta_t^g = \delta_t$ for all $t \neq j,i$. In particular, $\tilde{B}$ is not a base. The same happens if we consider $\tilde{B} = B' \setminus \{(\delta_{i+1},\lambda_i)\}$.
	\end{proof}
	
	We are currently not able to determine necessary and sufficient conditions that establish the exact value of $\varepsilon$ in Theorem~\ref{thrm:41}, see Remark~\ref{monday} for a ``running conjecture''.
	Nevertheless, if $G$ and $H$ are symmetric groups of degree $n$ and $m$ respectively, then we can compute exactly $\mathcal{M}(G \times H,\{1,\dots,n\} \times \{1,\dots,m\})$, as we shall see now. This also gives us examples where $\varepsilon$ is $0,1,2$ respectively, proving that each possibility in Theorem~\ref{thrm:41} can indeed occur, see Example~\ref{example1}.
	
For the rest of the proof, we let $S_n$ be the symmetric group on $\{1,\ldots,n\}$ and, for not making the notation too cumbersome, we write $b(S_n)=B(S_n)=b(S_n,\{1,\ldots,n\})=B(S_n,\{1,\ldots,n\})$.	
	\begin{lemma}\label{ProdSym}
		Let $G$ be a non-identity permutation group acting on $\Delta$ and let $n>1$. Then
		\[
			B(G\times S_n,\Delta\times \{1,\dots,n\}) \leq B(G,\Delta) + B(S_n) -1.
		\]
	\end{lemma}
	\begin{proof}
		Suppose that $B(G \times S_n,\Delta\times \{1,\dots,n\}) = B(G,\Delta)+ B(S_n)$. Let
		\[
		(\delta_1,\lambda_1),\dots, (\delta_k,\lambda_k)
		\] 
		 be a minimal base for the product action with $k=B(G,\Delta)+ B(S_n)$.
		Without loss of generality, we can reorder the base and suppose that $\lambda_1 = 1, \dots, \lambda_{n-1} = n-1$.
		Consider now the vector $v^{S_{n}}$. By lemma \ref{lemmaVect}, this should have $1$ in the first $n-1$ positions, and $0$ in the remaining. Now either $\lambda_j = n$ for all $j=n,\dots, k$, or there exists a $j=n,\dots,k$ such that $\lambda_j = n-s$, for $s=0,\dots, n-1$. In the first case, the vector $v^{S_{n}}$ would be equal to the zero vector, a contradiction. In the latter case, $v^{S_{n}}_{n-s}=0$, again a contradiction. 
	\end{proof}
	If we have two symmetric groups $S_{n}$ and $S_{m}$ with $n,m>1$, then, repeating the argument in the proof of Lemma~\ref{ProdSym}, we obtain  
	$B(S_{n}\times S_m,\{1,\dots,n \} \times \{1,\dots,m\}) \leq B(S_n) + B(S_m) -2.$
Now, from Theorem~\ref{thrm:41}, we deduce
$$B(S_{n}\times S_m,\{1,\dots,n \} \times \{1,\dots,m\})= B(S_n) + B(S_m) -2.$$
	\begin{proposition}\label{prodSym}
		Let $n_1, \dots, n_t$ be integer numbers greater than $1$. Then
		\[
			\mathcal{M}(S_{n_1}\times \dots \times S_{n_t},\{1,\ldots,n_1\}\times\cdots\times\{1,\ldots,n_t\})=\left\{\max_{i=1,\dots,t}(B(S_{n_i})), \dots, \sum_{i=1}^{t}B(S_{n_i})-t\right\}.
		\]
	\end{proposition}
	\begin{proof}
Set $\Omega=\{1,\ldots,n_1\}\times\cdots\times\{1,\ldots,n_t\}$.
Arguing by induction and using Theorem~\ref{thrm:41}, we deduce that $\mathcal{M}(S_{n_1}\times \dots \times S_{n_t},\Omega)$ is an interval of integer numbers and 
		\[
			b(S_{n_1}\times \dots \times S_{n_t},\Omega) = \max_{i=1,\dots,t}(b(S_{n_i})).
		\]
		Let us now focus on $B(S_{n_1}\times \dots \times S_{n_t},\Omega)$. Using Lemma~\ref{ProdSym} and induction, we have
		\begin{align*}
			B(S_{n_1}\times \dots \times S_{n_t},\Omega) &\leq B(S_{n_1}) + B(S_{n_2}\times \dots \times S_{n_t},\{1,\ldots,n_2\}\times\cdots\times\{1,\ldots,n_t\})-1 \\
			&\leq \sum_{i=1}^t B(S_{n_i}) - (t-1)-1 = \sum_{i=1}^t B(S_{n_i}) -t.
		\end{align*}
		To complete the proof, we establish a minimal base of cardinality $\sum_{i=1}^t B(S_{n_i})-t$. Indeed, the  minimal base is given by
		\begin{align*}
			\{&(1,1,\dots,1) ,(2,1,\dots,1) ,\dots, (n_1-2,1,\dots,1), \\ &(n_1-1,2,\dots,1), (n_1-1,3,\dots,1), \dots, (n_1-1,n_2-1,\dots,1), \\
			&\vdots \\
			&(n_1-1,n_2-1,\dots,2),\dots, (n_1-1,n_2-1,\dots,n_t-1)\}.\qedhere
		\end{align*}
	\end{proof}

\begin{example}\label{example1}{\rm
	Now we can give some examples with $\varepsilon=0,1,2$ respectively. Take $H=S_3 \times S_3$ in its product action on $\Omega=\Delta \times \Delta$, where $\Delta = \{1,2,3\}$. By Proposition \ref{prodSym}, we have $$B(H, \Delta \times \Delta) = 2,$$ so in this case $\varepsilon=2$. \\
	Consider the action of $H \times H$ on $\Omega \times \Omega$. Clearly, this is equivalent to the product action of $S_3\times S_3 \times S_3 \times S_3  $ on $\Delta\times \Delta \times \Delta \times \Delta $. So, by Proposition \ref{prodSym}, we have $$B(H \times H , \Omega \times \Omega)=B(S_3\times S_3 \times S_3 \times S_3,\Delta\times \Delta \times \Delta \times \Delta) = 2+2+2+2-4 = 4.$$
	This implies that $$4 = B(H\times H,\Delta\times \Delta) = B(H,\Delta)+B(H,\Delta)-\varepsilon = 2+2-\varepsilon = 4-\varepsilon,$$ And so $\varepsilon = 0$ in this case.\\
	Finally, take the group $S_n \times H$ acting on $\{1,\dots,n\} \times \Delta \times \Delta$. Then $$B(S_n \times H,\{1,\dots,n\}\times \Delta \times \Delta) = n-1+2-\varepsilon = n+1-\varepsilon.$$ Again from Proposition \ref{ProdSym} we have
	\[
		B(S_n \times H,\{1,\dots,n\}\times \Delta \times \Delta) = B(S_n \times S_3 \times S_3, \{1,\dots,n\}\times \Delta \times \Delta) = n-1+2+2-3 = n,
	\]
	implying that $\varepsilon = 1$ in this case.}
\end{example}

\begin{remark}\label{monday}{\rm
We observe that in Example~\ref{example1} we obtain $\varepsilon=1$ when the group $G=S_n\times H$ and $H$ itself is endowed of a product action. A similar comment applies for the example with $\varepsilon=0$ where $G=H\times H$ and both factors of this direct product are endowed of a product action.

Inspired by these examples, we say that a permutation group $X$ is \textbf{\textit{product indecomposable}} if there exist no non-identity permutation groups $A$ and $B$  such that $X$ is permutation equivalent to $A\times B$ with its product action.

We believe that, when $G$ and $H$ are product indecomposable then $\varepsilon=2$ in Theorem~\ref{thrm:41}. Similarly, we believe that $\varepsilon=0$ in Theorem~\ref{thrm:41} when $G$ and $H$ are not product indecomposable, and $\varepsilon=1$ in all remaining cases.}
\end{remark}	
	The following lemma is an easy generalization of \cite[Lemma 2.6]{GLSp}.
	\begin{lemma}\label{aux}
		Let $\{G_i\}_{i=1}^t$ be non-identity permutation groups, acting on $\Delta_i$ respectively. Then
		\[
			I(G_1\times\dots\times G_t,\Delta_1\times\dots\times\Delta_t) = \sum_{i=1}^{t}I(G_i,\Delta_i) - (t-1).
		\]
	\end{lemma}
	We are now ready to prove Theorem \ref{thrm:3}.
	\begin{proof}[Proof of Theorem~$\ref{thrm:3}$]Let $X$ be an interval of positive integers, not containing $1$.
		If $X = \{a\}$, then by taking $G = H = S_{a+1}$ we have $\mathcal{M}(G)=\mathcal{I}(H)=\{a\}=X$. Suppose now $|X|>1$ and let  $X = \{a,a+1,\dots,b\}$, with $a>1$.

		We start with the minimal bases. Divide $b$ by $a-1$: 
		\[
			b=t(a-1)+r
		\]
		with $0 \leq r < a-1$. We consider two cases.
		
		If $r=0$, that is $b=t(a-1)$, then take 
		\[
			G = \underbrace{S_{a+1} \times \dots \times S_{a+1}}_{t\text{-times}}
		\]
in its product action on $\Omega=\{1,\ldots,a+1\}\times\cdots \times\{1,\ldots,a+1\}$.
		From Proposition~\ref{prodSym}, $b(G,\Omega) = a$ and 
		\[
			B(G,\Omega) = \sum_{i=1}^{t}B(S_{a+1}) -t= ta -t = t(a-1) = b. 
		\]
		
		If $r>0$, take $k=r+2$, so that $b=t(a-1)+(k-2)$. Since $r<a-1$, then $k-1 < a$. Moreover, $k>2$. Consider the group
		\[
			G=\underbrace{S_{a+1} \times \dots \times S_{a+1}}_{t\text{-times}} \times S_{k}
		\]
in its product action on $\Omega=\{1,\ldots,a+1\}\times\cdots \times\{1,\ldots,a+1\}\times\{1,\ldots,k\}$. Arguing as above, we deduce $b(G,\Omega) = \max(a,k-1) = a$, and 
		\[
			B(G,\Omega) = at+k-1-t-1 = t(a-1) + k-2 = b,
		\]
		as desired.
		
		For the irredundant bases, again divide $b$ by $a-1$: $b=t(a-1)+r$. Now if $r=0$ take $H=G$ and $\Omega$ as before and use Lemma~\ref{aux}. If $r>0$, take $k=r+1$ and
		\[
			H=\underbrace{S_{a+1} \times \dots \times S_{a+1}}_{t\text{-times}} \times S_{k}
		\]
in its product action on $\Omega=\{1,\ldots,a+1\}\times\cdots \times\{1,\ldots,a+1\}\times\{1,\ldots,k\}$.
		Then, we have $b(H) = \max(a,k-1) = a$ and, by Lemma~\ref{aux}, we have
		\[
			I(H,\Omega) = at + k-1 - t = t(a-1) + k-1 = t(a-1)+r = b.\qedhere
		\]
	\end{proof}
	\section{Special cases for $\mathcal{M}(G,\Omega)$}\label{SpecCase}

	Theorem \ref{thrm:3} shows that any interval of natural numbers is the set of cardinalities of minimal bases for some transitive group. This could suggest that, if $G$ is a transitive group, then $\mathcal{M}(G,\Omega)$ is an interval, but this is not the case. 
	
	\begin{proposition}\label{number prop}
	Let $X = \{3,n\}$, with $n>3$. Then, there exists a transitive group $G$ on $\Omega$ with $\mathcal{M}(G,\Omega) = X$.
	\end{proposition}
	\begin{proof}
Let $S_n$ be the symmetric group on $n$ symbols endowed of its natural action on $\{1,\ldots,n\}$, let $G=S_{n} \mathrm{wr} \langle\sigma\rangle$ where $\sigma$ has order $3$, let $N = S_n \times S_n \times S_n \leq G$ and let
		\[
			H = S_{n} \times (S_{n})_n \times 1 \leq N \leq G,
		\] 
		where $(S_{n})_n = S_{n-1}$ is the stabilizer of the  point $n$ in $S_{n}$. We consider the action of $G$ on the set $\Omega$ of right cosets of $H$.  As $H$ is core-free in $G$, we may view $G$ as a permutation group on $\Omega$. We claim that $\mathcal{M}(G,\Omega)=\{3,n\}$.

Set $\omega=H\in \Omega$ and observe  that $G_\omega = H$. We can partition $G$ into three $N$ cosets, namely, $G = N \cup \sigma N \cup \sigma^2 N.$ This partition, in turn, partitions $\Omega$ into three $N$-orbits:
$$\Omega=\omega^N\cup\omega^{\sigma N}\cup\omega^{\sigma^2N}.$$ If $x = (x_1,x_2,x_3) \in N$ and $ y = (y_1,y_2,y_3) \in N$, we have
		\begin{align}\label{tue0}
			G_\omega &= S_{n} \times (S_{n})_n \times 1, \\\nonumber
			G_{\omega^{\sigma x}} & =1 \times S_{n} \times (S_{n})_{n^{x_3}}, \\\nonumber
			G_{\omega^{\sigma^2 y}} &= (S_{n})_{n^{y_1}} \times 1 \times S_{n}.
		\end{align}
		It follows that
		\[
			G_\omega \cap G_{\omega^{\sigma x}} \cap G_{\omega^{\sigma^2 y}} = 1,
		\]
		and so, for every choice of $x,y \in N$, $\{ \omega, \omega^{\sigma x}, \omega^{\sigma^2 y}\}$ is a minimal base for the action of $G$ on $\Omega$. Therefore, $3\in \mathcal{M}(G,\Omega)$.
		
		We now construct a minimal base of cardinality $n$. First we again take $\{\omega, \omega^{\sigma x}\}$ for some $x \in N$, so that the pointwise stabilizer of this subset is
		\begin{equation}\label{tue1}
			G_\omega \cap G_{\omega^{\sigma x}} =   1 \times (S_n)_n \times 1.	
		\end{equation}
		If we choose a point from $\omega^{\sigma^2 N}$, then  we find again a minimal base of cardinality $3$. So we are forced to choose the remaining points only from the first two $N$-orbits. But if we choose a point from the second $N$-orbit, say $\omega^{\sigma y}$ for $y \in N$, from~\eqref{tue0} and~\eqref{tue1}, we obtain $$G_\omega \cap G_{\omega^{\sigma x} } \cap G_{\omega^{\sigma y}} = G_\omega \cap G_{\omega^{\sigma x} }.$$ Therefore, we need to take points only from the first $N$-orbit. For $y \in N$, we have
		\[
			G_{\omega^{y}} = S_n \times (S_n)_{n^{y_2}} \times 1 ,
		\] 
		and so
		\[
			G_\omega \cap G_{\omega^\sigma x} \cap G_{\omega^{y}} =
			   1 \times (S_n)_{n,n^{y_2}} \times 1.	
		\]
		In this way, we need $n-1$ points of the form $\omega^y$ to reach the identity, and so we have constructed a minimal base of cardinality $n$. Therefore, $n\in\mathcal{M}(G,\Omega)$.
		
		Given that we have exhausted our options for selecting orbits to form the basis points, we can deduce that there are no other cardinalities for a minimal base in $G$. Consequently, $\mathcal{M}(G, \Omega)$ is contained in $\{3, n\}$, marking the conclusion of our proof.
	\end{proof}
	In a very similar way, we can also prove the following proposition.
	\begin{proposition}
		Let $X = \{2,n\}$, with $n>3$. Then, there exists a transitive group $G$ on $\Omega$ with $\mathcal{M}(G,\Omega) = X$.
	\end{proposition}
	\begin{proof}
		Take $G =S_{n} \mathrm{wr} C_2$ acting on the right cosets of $$H = (S_{n})_n \times 1\leq G.$$
		We skip the proof, as it mirrors the steps taken in proving Proposition~\ref{number prop}.
	\end{proof}
	This process can not be repeated to produce other transitive permutation groups $G$ on $\Omega$ with $|\mathcal{M}(G,\Omega)|=2$. Indeed, with the same argument as in Proposition~\ref{number prop} it is possible to prove the following proposition.
	\begin{proposition}
		Let $n>4$ and let $G = S_n \mathrm{wr} C_4$ be endowed of its action on the set $\Omega$ of right cosets of $H = S_n \times S_n \times S_{n-1} \times 1 \leq G$.
		Then $\mathcal{M}(G,\Omega) =  \{4, n, 2n-2\}$. 
	\end{proposition}

\section{The set $\mathcal{M}(G,\Omega)$  for primitive groups}\label{last}
Other than computational evidence, there is no additional confirmation that the set $\mathcal{M}(G,\Omega)$ forms an interval for primitive groups $G$. Nevertheless, 
we hold a belief in a certain level of regularity within primitive groups, prompting us to propose Conjecture~\ref{conj2}. Despite this, we do not dare to conjecture that, for each interval $X$ of natural numbers, not containing $1$, there exists a primitive group $G$ on $\Omega$ with $X=\mathcal{M}(G,\Omega)$. We leave this as yet another conjecture: we phrase it in terms of irredundant bases.
\begin{conjecture}\label{conj2_2}
There exists an interval  $X$ of positive integers, not containing $1$, such that no primitive permutation group $G$ on $\Omega$ satisfies $X=\mathcal{I}(G,\Omega)$.
	\end{conjecture}
	As an evidence of this conjecture, we consider the symmetric group $S_n$ in its natural action on the set $\Omega_k$ of $k$-subset of $\{1,\dots,n\}$ with $1 \leq k \leq \frac{n}{2}$. \\
	The base size of this action has been recently obtained in two independents works, \cite{DVCR} and \cite{MSp}. In particular, in our example, we will use Table $1$ from \cite{MSp} the read off $b(S_n,\Omega_k)$. Note that, before these two papers, some partial results about $b(S_n, \Omega_k)$ was given by Halasi in \cite{halasi}. Indeed, if $n \geq k^2$, then
	\[
		b(S_n,\Omega_k) = \bigg \lceil \frac{2(n-1)}{k+1} \bigg \rceil. 
	\]

	 On the other side, in \cite[Theorem 1.1]{LG}, it is proved that
	\[
		I(S_n,\Omega_k) = 
		\begin{cases}
			n-1, \text{ if } \gcd(n,k)=1 \\
			n-2, \text{ otherwise.}
		\end{cases}
	\]
	Combining these results, we can prove that there exist no $n,k$, with $ k \leq n/2$, such that 
	\[
		\mathcal{I}(S_n,\Omega_k)=\{3,4,\dots,12\}.
	\]
	To prove that, we study two cases. Firstly, suppose that $\gcd(n,k) = 1$. Then,
	\[
		12 = I(S_n,\Omega_k) = n-1,
	\]
	and so $n=13$. However, from~\cite[Table~1]{MSp}, we deduce
	\begin{align*}
		&b(S_{13},\Omega_1) = 12, \\
		&b(S_{13},\Omega_2) = 8, \\
		&b(S_{13},\Omega_3) = 6, \\
		&b(S_{13},\Omega_4) = 5, \\
		&b(S_{13},\Omega_5) = 5, \\
		&b(S_{13},\Omega_6) = 4.
	\end{align*}
	Suppose now that $\gcd(n,k) > 1$. Then 
	\[
		12 = I(S_n,\Omega_k) = n-2,
	\]
	and so $n=14$. In this case, from~\cite[Table~1]{MSp}, we deduce
	\begin{align*}
		&b(S_{14},\Omega_2) = 9, \\
		&b(S_{14},\Omega_4) = 6, \\
		&b(S_{14},\Omega_6) = 5, \\
		&b(S_{14},\Omega_7) = 4.
	\end{align*}
	This shows that the interval $\{3,\dots,12\}$ does not arise as $I(S_n,\Omega_k)$ for some $n,k$.

\end{document}